\documentclass[12pt]{amsart}
\usepackage{amsmath,amssymb,amsthm,mathabx} 
\usepackage[all]{xy}
\CompileMatrices
\usepackage{verbatim}
\title{On the Diophantine equation $p^x + p^y = z^{2n}$}
\author{Dibyajyoti Deb}
\address{Department of Mathematics, Oregon Institute of Technology, Klamath Falls, OR, USA}
\email{dibyajyoti.deb@oit.edu}

\theoremstyle{plain}
\newtheorem{theorem}{Theorem}[section]
\newtheorem{lemma}[theorem]{Lemma}

\theoremstyle{definition}

\begin{document} 

\begin{abstract}
In \cite{TS1}, Tatong and Suvarnamani explores the Diophantine equation $p^x + p^y = z^2$ for a prime number $p$. In that paper they find some solutions to the equation for $p=2, 3$. In this paper, we look at a general version of this equation and solve it completely. 
\end{abstract}

\maketitle

\section{Introduction} 

Diophantine equations of the form $a^x + b^y = z^2$ have been studied in recent papers by various authors. In this paper we use elementary methods to completely solve the equation $p^x + p^y = z^{2n}$ in non-negative integers where $p$ is a prime number. 

\section{Preliminaries}
\begin{theorem}\label{catalan}
The Diophantine equation $a^x - b^y = 1$ has a unique solution in integers $a, b, x,$ and $y$ with min$\{a, b, x, y \} > 1$. This solution is given by $(3, 2, 2, 3)$.
\end{theorem}

\begin{proof}
This is the famous Catalan's Conjecture (which can be also called Mihailescu's Theorem) that was proven by Mihailescu. For the proof see \cite{Mih}.
\end{proof}

\begin{lemma}\label{lem2}
For an odd prime $p > 3$, the Diophantine equation $p^x + 1 = z^2$ has no non-negative integer solutions.
\end{lemma}

\begin{proof}
On the contrary, suppose that there are non-negative integer solutions $(p, x, z)$. If $x = 0$, then $z^2 = 2$ which is not solvable in integers. Therefore, $x \geq 1$. Since $p > 3$, therefore, $z^2 = p^x + 1 \geq 5^1 + 1 = 6$. Thus, $z \geq 3$. We rewrite the equation as $z^2 - p^x = 1$. Since $p > 3$, hence by Theorem \ref{catalan}, $x = 1$. This means $z^2 = 1 + p$ or $(z-1)(z+1) = p$, which is impossible as $p$ is a prime  number and $z \geq 3$. Hence, when $p > 3$, the equation $p^x + 1 = z^2$ has no non-negative integer solutions.
\end{proof}

\section{Main Result}

We now present our main result.

\begin{theorem} \label{mainthm}
If $p$ is a prime number and $n \geq 1$ is an integer, then the Diophantine equation 
\begin{equation*}
p^x + p^y = z^{2n}
\end{equation*}
has the following set of solutions $(x, y, z)$ in non-negative integers.

\noindent
For $n=1$,
$$(x, y, z) =
\begin{cases}
(2s+3, 2s, 3 \cdot 2^s), (2s, 2s+3, 3 \cdot 2^s), (2s+1, 2s+1, 2^{s+1}), & p = 2, \text{ and } s \in \mathbb{N} \cup \{0\} \\
(2s+1, 2s, 2 \cdot 3^s), (2s, 2s+1, 2 \cdot 3^s), & p = 3, \text{ and } s \in \mathbb{N} \cup \{0\} \\
\text{ No solutions}, & p > 3
\end{cases}
$$
For $n>1$,
$$(x, y, z) =
\begin{cases}
(2s+1, 2s+1, 2^{(s+1)/n}), & p = 2, s \in \mathbb{N} \cup \{0\} \text{ and } s+1 \equiv 0 \text{ mod } n   \\
\text{ No solutions}, & p \geq 3
\end{cases}
$$
\end{theorem}

\noindent
The proof of this theorem combines the proofs of the next two theorems.

\begin{theorem}\label{thmequal1}
If $p$ is a prime number then the Diophantine equation 
\begin{equation}\label{maineq}
p^x + p^y = z^2
\end{equation}
has the following set of solutions $(x, y, z)$ in non-negative integers.

\noindent
$$(x, y, z) =
\begin{cases}
(2s+3, 2s, 3 \cdot 2^s), (2s, 2s+3, 3 \cdot 2^s), (2s+1, 2s+1, 2^{s+1}), & p = 2, \text{ and } s \in \mathbb{N} \cup \{0\} \\
(2s+1, 2s, 2 \cdot 3^s), (2s, 2s+1, 2 \cdot 3^s), & p = 3, \text{ and } s \in \mathbb{N} \cup \{0\} \\
\text{ No solutions}, & p > 3
\end{cases}
$$
\end{theorem}

\begin{proof}
This is the case $n= 1$ of Theorem \ref{mainthm}. We consider several cases.

\noindent
{\em Case 1 :} $x = y$. In this case, 
$$p^x + p^y = 2p^x = z^2$$
Thus $z = 2^{1/2}p^{x/2}$. As $p$ is a prime number and $z$ is a non-negative integer, therefore, $p = 2$ and $x$ is an odd non-negative integer. Thus, when $x = 2s+1$ for any non-negative integer $s$, we have $z = 2^{s+1}$. Hence, when $y=x$, the solution set $(x, y, z)$ consists of 3-tuples of the form $(2s+1, 2s+1, 2^{s+1})$ for any $s \in \mathbb{N} \cup \{0\}$.

\enspace

\noindent
{\em Case 2 :} $x \neq y$ and $x<y$. 
In this case,
$p^x(1+p^{y-x}) = z^2$, which implies that $p \mid z$. Let $e$ be the highest power of $p$ that divides $z$, i.e., $p^e \mid z$ but $p^{e+1} \notdivides z$. Suppose $z = p^e k$ with $p \notdivides k$. We then have,
$$p^x(1+p^{y-x}) = (p^e k )^2 = p^{2e}k^2$$
Since $p$ is a prime number and $p \notdivides k$, therefore, $p^x = p^{2e}$ which implies $x = 2e$. We then have,
\begin{equation}\label{eq1}
k^2 = 1 + p^{y-x} = 1 + p^{y-2e}
\end{equation}
We now divide this problem into several cases.

\enspace

\noindent 
{\em Case 2.1 :} $p = 2$ and $y-2e = 1$. In this case, 
$$k^2 = 1 + 2^{y-2e} = 1+ 2 = 3$$
This has no integer solutions in $k$, hence the Diophantine equation $p^x + p^y = z^2$ has no non-negative integer solutions when $p=2$ and $y-x = 1$.

\enspace

\noindent
{\em Case 2.2 :} $p = 2$ and $y-2e > 1$. In this case, since $p=2$ and $y-2e > 1$, therefore, $k^2 = 1 + 2^{y-2e} \geq 5$. Therefore, $k \geq 3$. Rewriting Equation (\ref{eq1}) we have,
$$k^2 = 1 + 2^{y-2e}$$
\begin{equation}\label{eq2}
k^2 - 2^{y-2e} = 1 
\end{equation}
As min$\{k, 2, 2, y-2e \} > 1$, therefore, Equation (\ref{eq2}) has a unique solution $(k, y-2e)$ given by $(3, 3)$ by Theorem \ref{catalan}. Thus, we have $x = 2e, y = 2e+3$ and $z = 2^e \cdot 3$. Hence, when $p = 2$ and $y-x = y - 2e >1$, the solution set $(x, y, z)$ of Equation (\ref{maineq}) consists of 3-tuples of the form $(2s, 2s+3, 3 \cdot 2^s)$ for any $s \in \mathbb{N} \cup \{0\}$.

\enspace

\noindent
{\em Case 2.3 :} $p=3$ and $y-2e = 1$. In this case we have,
$$k^2 = 1 + 3^{y-2e} = 1 + 3 = 4$$
This implies $k = 2$. Thus, $x=2e, y = 2e+1$ and $z = 2 \cdot 3^e$. Thus when $p=3$ and $y-x = y - 2e =1$, the solution set $(x, y, z)$ of Equation (\ref{maineq}) consists of 3-tuples of the form $(2s, 2s+1, 2 \cdot 3^s)$ for any $s \in \mathbb{N} \cup \{0\}$.

\enspace

\noindent
{\em Case 2.4 :} $p = 3$ and $y-2e > 1$. In this case, since $p=3$ and $y-2e > 1$, therefore, $k^2 = 1 + 3^{y-2e} \geq 10$. Therefore, $k \geq 4$. Rewriting Equation (\ref{eq1}) we have,
$$k^2 = 1 + 3^{y-2e}$$
\begin{equation}\label{eq3}
k^2 - 3^{y-2e} = 1 
\end{equation}
As min$\{k, 2, 3, y-2e \} > 1$, therefore, Equation (\ref{eq3}) does not have any non-negative integer solutions by Theorem \ref{catalan}. Thus when $p=3$ and $y-x = y - 2e > 1$, Equation (\ref{maineq}) has no solutions in non-negative integers.

\enspace

\noindent
{\em Case 2.5 :} $p >3$. Looking at Equation (\ref{eq2}),
$$1+p^{y-2e} = k^2$$
The equation has no solutions in non-negative integers by Lemma \ref{lem2}. Thus, when $p>3$, Equation (\ref{maineq}) has no solutions in non-negative integers.

Thus, taking {\em Cases 2.1-2.5} into account we see that the set of solutions $(x, y, z)$ of Equation (\ref{maineq}) when $x < y$ is given by,
$$(x, y, z) =
\begin{cases}
(2s, 2s+3, 3 \cdot 2^s), (2s+1, 2s+1, 2^{s+1}), & p = 2, \text{ and } s \in \mathbb{N} \cup \{0\} \\
(2s, 2s+1, 2 \cdot 3^s), & p = 3, \text{ and } s \in \mathbb{N} \cup \{0\} \\
\text{ No solutions}, & p > 3
\end{cases}
$$

\enspace

\noindent
{\em Case 3 :} $x \neq y$ and $x>y$. 
This case is similar to {\em Case 2} and its various cases. The only difference results in switching the values of $x$ and $y$ in our final answer. Thus the set of solutions $(x, y, z)$ of Equation (\ref{maineq}) when $x > y$ is given by,
$$(x, y, z) =
\begin{cases}
(2s+3, 2s, 3 \cdot 2^s), (2s+1, 2s+1, 2^{s+1}), & p = 2, \text{ and } s \in \mathbb{N} \cup \{0\} \\
(2s+1, 2s, 2 \cdot 3^s), & p = 3, \text{ and } s \in \mathbb{N} \cup \{0\} \\
\text{ No solutions}, & p > 3
\end{cases}
$$
This concludes the proof.

\noindent
 
\end{proof}

\noindent
Now we consider the case $n > 1$.
\begin{theorem} \label{thmgreater1}
If $p$ is a prime number and $n > 1$ is an integer, then the Diophantine equation 
\begin{equation}\label{eqngreater1}
p^x + p^y = z^{2n}
\end{equation}
has the following set of solutions $(x, y, z)$ in non-negative integers.

\noindent
$$(x, y, z) =
\begin{cases}
(2s+1, 2s+1, 2^{(s+1)/n}), & p = 2, s \in \mathbb{N} \cup \{0\} \text{ and } s+1 \equiv 0 \text{ mod } n   \\
\text{ No solutions}, & p \geq 3
\end{cases}
$$
\end{theorem}
\begin{proof}
{\em Case 1:} $p = 2$.
Suppose that there exists non-negative integers $x, y$ and $z$ that satisfies Equation (\ref{eqngreater1}). Let $w=z^n$. In that case $(x, y, z^n)$ is a solution in non-negative integers of the equation 
\begin{equation}\label{eq4}
2^x + 2^y = w^2
\end{equation}
By Theorem \ref{thmequal1} the only solutions $(x, y, w)$ to Equation (\ref{eq4}) are $(2s+3, 2s, 3 \cdot 2^s), (2s, 2s+3, 3 \cdot 2^s), (2s+1, 2s+1, 2^{s+1})$ for $s \in \mathbb{N} \cup \{0\}$. This means either $w = 3 \cdot 2^s$ or $w = 2^{s+1}$. 

\enspace

\noindent
{\em Case 1.1 :} $w = 3 \cdot 2^s$. This is not possible as $w = z^n = 3 \cdot 2^s$ which is not solvable in integers $z$ for $n > 1$.

\enspace

\noindent
{\em Case 1.2 :} $w = 2^{s+1}$. This means $z^n = 2^{s+1}$ or $z = 2^{(s+1)/n}$. If $n \mid (s+1)$, then $z$ is an integer and we have solutions $(2s+1, 2s+1, 2^{(s+1)/n})$ for Equation (\ref{eqngreater1}).

\enspace

\noindent
{\em Case 2 :} $p \geq 3$. We proceed similarly to {\em Case 1}. Suppose that there exists non-negative integers $x, y$ and $z$ that satisfies Equation (\ref{eqngreater1}). Let $w=z^n$. In that case $(x, y, z^n)$ is a solution in non-negative integers of the equation 
\begin{equation}\label{eq5}
p^x + p^y = w^2
\end{equation}

\enspace

\noindent
{\em Case 2.1 :} $p = 3$. By Theorem \ref{thmequal1}, the only solutions $(x, y, w)$ to Equation (\ref{eq5}) are $(2s+1, 2s, 2 \cdot 3^s), (2s, 2s+1, 2 \cdot 3^s)$ for $s \in \mathbb{N} \cup \{0\}$. This means $w = 2 \cdot 3^s$ for both solutions. However, $z^n = 2 \cdot 3^s$ doesn't have any solutions in integers $z$ for $n > 1$. Therefore   Equation (\ref{eq4}) has no solutions in non-negative integers.

\enspace

\noindent
{\em Case 2.2 :} $p >3$. By Theorem \ref{thmequal1}, Equation (\ref{eq4}) has no solutions $(x, y, w)$ in non-negative integers. Hence Equation (\ref{eq5}) has no solutions $(x, y, z)$ in non-negative integers.

\enspace

\noindent
Combining the results from {\em Case 1} and {\em Case 2} we see the results of Theorem \ref{thmgreater1}.
\end{proof}

\section{Conclusion}
We see that the Diophantine equation $p^x + p^y = z^{2n}$ has infinitely many solutions when $n=1$ and $p=2$ or $p=3$. The equation has no solutions when $p>3$. It also has infinitely many solutions when $n > 1$ for $p=2$. However, the equation does not have any solutions when $p \geq 3$. All these solutions are given in the statement of Theorem \ref{mainthm}.

\end{document}